%% file: gascom-paper-kbv.tex
\documentclass[a4paper,12pt]{article}
\usepackage[english]{babel}
\usepackage[T1]{fontenc}
\usepackage[utf8]{inputenc}
\usepackage[obeyspaces]{url}
\usepackage[dvipsnames]{xcolor}

\usepackage{mathtools, amsmath, amsthm, amssymb, mathtools,
  tikz,tikz-cd, subcaption, framed, float, tensor, bm, dsfont,
  floatpag, microtype, fullpage}

\usepackage[normalem]{ulem}
\usepackage{thm-restate}

\usepackage[
  bookmarks=false, 
  colorlinks,
  citecolor=blue,
  linkcolor=blue,
  urlcolor=blue,
]{hyperref}
\usepackage[affil-it]{authblk}
\usepackage{cleveref}
\usetikzlibrary{shapes.geometric, backgrounds, calc,
  arrows, automata, fit, positioning, patterns,
  decorations.pathreplacing, tikzmark, arrows.meta}

\newtheorem{thm}{Theorem}
\newtheorem{lemma}{Lemma}
\newtheorem{cor}{Corollary}

\newtheorem{prop}{Proposition}

{
\theoremstyle{definition}

}

{

}

\newcommand{\Q}{\mathbb{Q}}

\def\ONE{\mathds{1}}

\newcommand{\B}{\mathcal{B}}

\let\ge\geqslant
\let\leq\leqslant
\let\geq\geqslant

\DeclareUnicodeCharacter{21A6}{$\mapsto$}

\author{Jean-Luc Baril}
\author{Sergey Kirgizov}
\author{Vincent Vajnovszki}
\affil{\rm LIB, Université de Bourgogne Franche-Comté\protect\\
  B.P. 47 870, 21078 Dijon Cedex France\protect\\
   {\tt E-mails: \{barjl, sergey.kirgizov, vvajnov\}@u-bourgogne.fr
   }
}
\date{\today}

\title{Asymptotic bit frequency in Fibonacci words}

\begin{document}

\maketitle

\begin{abstract}
It is known that binary words containing no $k$ consecutive 1s are
enumerated by $k$-step Fibonacci numbers. In this note we discuss
the expected value of a random bit in a random word of length $n$
having this property. 
  
\end{abstract}

\section{Introduction}
For $n\geq 0$ and $k\geq 2$, we denote by $\B_n(1^k)$ the set of length $n$ binary words avoiding $k$
consecutive 1s. For example, we have

$$\begin{aligned}
  \B_4(11) & = \{ 0000, 0001, 0010, 0100, 0101, 1000, 1001, 1010 \}, \mbox{ and } \\
  \B_4(111) & = \{ 0000, 0001, 0010, 0011, 0100, 0101, 0110, 1000,
                   1001, 1010, 1011, 1100, 1101 \}.
\end{aligned}$$
It is well known, see Knuth~\cite[p. 286]{knuth3}, that $\B_n(1^k)$ is
enumerated by  the $k$-step Fibonacci numbers, precisely
$|\B_n(1^k)|=f_{n+k,k}$, where $f_{n,k}$ is defined,
following Miles~\cite{miles} as
\begin{equation*}
f_{n,k} = \begin{cases}
0 & \text{if } 0 \leq n \leq k - 2,\\
1 & \text{if }  n = k - 1,\\
\sum_{i=1}^k f_{n-i,k} & \text{otherwise}.
\end{cases}
\label{fib_seq}
\end{equation*}

Denote by $v_{n,k}$ the {\em frequency} (also called {\em popularity})
of 1s in $\B_n(1^k)$, i.e. the total number of 1s in all words of
$\B_n(1^k)$. For instance, $v_{4,2} = 10$ and $v_{4,3} = 22$.  The
ratio of frequency of 1s to the overall number of bits in words of
$\B_n(1^k)$ is $$\alpha_{n,k} = \frac{v_{n,k}}{n \cdot |\B_n(1^k)|},$$
and it equals the expected value of a random bit in a random word from
$\B_n(1^k)$. In~\cite{ourfibo}, the authors left without proof the
fact that, for any $k\geq 2$, $\lim_{n \to \infty} \alpha_{n,k}$
converges to a non-zero value as $n$ grows.  This note is devoted to
proving this fact, which apart from its interest {\em en soi} has
practical counterparts.  Indeed, words in $\B_n(1^k)$ play a critical
role in some telecommunication frame synchronization protocols, see
for example~\cite{Bajic,BBPV,CKPW}, or in particular Fibonacci-like
interconnection networks~\cite{EI}.

Our discussion is based on the bivariate generating function

$$
F_k (x,y) = \sum_{n=0}^\infty \sum_{m=0}^{n - \left\lfloor \frac{n}{k} \right\rfloor} a_{n,m} x^n y^m
$$ whose coefficient $a_{n,m}$ equals the number of words from
$\B_n(1^k)$ containing exactly $m$ 1s. For $k=2$ and $k=3$, Table~\ref{tab} presents some values of $a_{n,m}$ for small $n$ and $m$.

\begin{table}[ht]
\begin{tabular}{c|ccccccccc} 
 $ m \backslash n $ & \bf 1 & \bf 2 & \bf 3 & \bf 4 & \bf 5 & \bf 6 & \bf 7 & \bf 8 & \bf 9  \\\hline 
\bf 0 & 1 & 1 & 1 & 1 & 1 & 1 & 1 & 1 & 1 \\
\bf 1 & 1 & 2 & 3 & 4 & 5 & 6 & 7 & 8 & 9 \\
\bf 2 &  &  & 1 & 3 & 6 & 10 & 15 & 21 & 28 \\
\bf 3 &  &  &  &  & 1 & 4 & 10 & 20 & 35 \\
\bf 4 &  &  &  &  &  &  & 1 & 5 & 15 \\
\bf 5 &  &  &  &  &  &  &  &  & 1
\end{tabular}\quad
\begin{tabular}{c|ccccccccc} 
 $ m \backslash n $ & \bf 1 & \bf 2 & \bf 3 & \bf 4 & \bf 5 & \bf 6 & \bf 7 & \bf 8 & \bf 9 \\\hline 
\bf 0 & 1 & 1 & 1 & 1 & 1 & 1 & 1 & 1 & 1  \\
\bf 1 & 1 & 2 & 3 & 4 & 5 & 6 & 7 & 8 & 9  \\
\bf 2 &  & 1 & 3 & 6 & 10 & 15 & 21 & 28 & 36  \\
\bf 3 &  &  &  & 2 & 7 & 16 & 30 & 50 & 77  \\
\bf 4 &  &  &  &  & 1 & 6 & 19 & 45 & 90  \\
\bf 5 &  &  &  &  &  &  & 3 & 16 & 51  \\
\end{tabular}
\caption{First few values of $a_{n,m}$ for $k = 2$ (left) and $k = 3$.}
\label{tab}
\end{table}

\section{Main result}

Proposition~\ref{gf} gives the expression of the generating function
$F_k (x,y)$. Even though this result is already obtained
in~\cite{ourfibo}, in order to make the paper self-contained we give
an alternative proof of it.  Then we calculate the generating
functions for the frequency of 1s and for the overall number of bits
in $\B_n(1^k)$ by means of classic generating functions manipulations
(Propositions \ref{pop1}). Applying Theorem 4.1 from~\cite{book},
after ensuring that its conditions are satisfied, we obtain the main
result of this note, Theorem~\ref{main}.  The evolution of the random
bit expectation for $k = 2$ and $k = 3$ is presented on
Figure~\ref{expectation} for small values of $n$. And numerical
estimations for the limit value ($n \to \infty$) of the random bit
expectation, for small values of $k$ are given in Table~\ref{tab2}.

\begin{figure}[ht]
  \centering
  \scalebox{0.9}{\input{expected}}
  \caption{Expected value of a random bit in a random word from $\B_n(1^2)$ (left)
    and  $\B_n(1^3)$ for small values of $n$.}
  \label{expectation}
\end{figure}
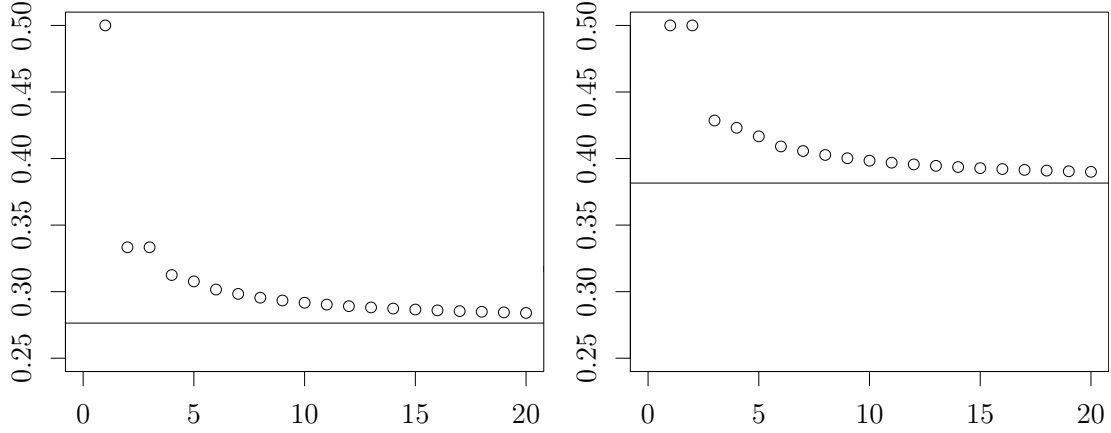

\begin{table}[ht]
  \centering
\raisebox{-.5\height}{\input{freq}}
\begin{tabular}{c|c|c} 
  $ k $ & Limit of the expected bit value \\\hline
2 & 0.276393202250021 \\
3 & 0.381580077680607 \\
4 & 0.433657112297348 \\
5 & 0.462073883180840 \\
6 & 0.478227505713290 \\
7 & 0.487545982771861 \\
8 & 0.492928265543398 \\
9 & 0.496019724266083 \\
10 & 0.497779940783496 \\
11 & 0.498772398758879 \\
12 & 0.499326557312936 \\
13 & 0.499633184444604 \\
\end{tabular}
\caption{Numerical estimations for the limit of the
  expected value of a random bit in a random word from $\B_n(1^k), n
  \to \infty$.}
\label{tab2}
\end{table}

\begin{prop}[\cite{ourfibo}]
  $$F_k(x,y)=\frac{y\left(1-(xy)^k\right)}{y-xy^2-xy+(xy)^{k+1}}.$$
  \label{gf}
\end{prop}
\begin{proof}

The set $\B(1^k) = \bigcup_{n=0}^\infty \B_n(1^k)$ respects the
following recursive decomposition
$$\B(1^k)=\ONE_{k-1} \cup
\Bigg(
\bigcup_{i=0}^{k-1} \bigg( 1^i0\cdot \B(1^{k}) \bigg)
\Bigg)
$$ where $\ONE_{k-1} = \bigcup_{i=0}^{k-1} \{1^i\}$ is the set of words in $\B(1^k)$ containing no 0s, and $\cdot$ denotes the
concatenation. Note that the empty word also lies in
$\ONE_{k-1}$.  The claimed generating function is the solution of the
following functional equation
$$F_k(x,y)=\sum_{i=0}^{k-1}x^iy^i+ F_k(x,y)\sum_{i=0}^{k-1}x^{i+1}y^{i}.$$
\end{proof}

In the proof of Theorem \ref{main} we need the following easy to derive results.

\begin{prop} $ $

\begin{itemize}
\item[$\bullet$] $P_k(x) = \frac{\partial F_k(x,y)}{\partial y} \vert_{y = 1}
$ is the generating function where the coefficient of $x^n$ is the frequency of 1s in $\B_n(1^k)$. We have 
$$P_k(x)=\frac{
      x  \cdot \sum_{i=0}^{k - 2} (i + 1) x^i 
}{\left(
  x^k + x^{k-1} + \cdots + x^2 + x - 1
  \right)^2
}.$$
\item[$\bullet$]
 $
  T_k(x) =x \frac{\partial F_k(x,1)}{\partial x} $ is the generating function where the coefficient of
  $x^n$ equals the total number of all bits in $\B_n(1^k)$.
We have
$$T_k(x) =\frac{
      x  \left( \sum_{i=0}^{k - 2} (2 i + 2) x^i
      + \sum_{i=k - 1}^{2 k - 2} (2k - i - 1) x^i \right)
  }
  {\left(
  x^k + x^{k-1} + \cdots + x^2 + x - 1
  \right)^2
}.
$$
\end{itemize}
\label{pop1}
\end{prop}

 Every root $r$ of a polynomial $h(x)$ of degree $n$ with a non-zero
 constant term corresponds to the root $1/r$ of its negative
 reciprocal $-x^n h(1/x)$.  The denominator of both $P_k(x)$ and
 $T_k(x)$ involves $x^k + x^{k-1} + \cdots + x^2 + x - 1$ and its
 negative reciprocal is $x^k - x^{k-1} - \cdots - x^2 - x - 1$ which
 is known in the literature as Fibonacci polynomial, see for
 instance~\cite{cipu, dubeau, flores, gross, hare, martin, miles,
   miller, wolf} and references therein.  In particular, Dubeau
 proved~\cite[Theorem 1]{dubeau} that its root of the largest modulus
 is $\varphi_k = \lim_{n\to \infty}f_{n+1,k}/f_{n,k}$, the generalized
 golden ratio, and $\varphi_k$ approaches 2 when $k \to
 \infty$~\cite[Theorem 2]{dubeau}.  Wolfram~\cite[Lemma 3.6]{wolf}
 showed that any other root $r$ of the Fibonacci polynomial satisfies
 $3^{-1/k} < |r| < 1$. See Figure~\ref{racines} for an illustration of
 this fact.  Moreover, Corollary 3.8 in \cite{wolf} proves that
 Fibonacci polynomial is irreducible over~$\Q$.  In order to refer
 later to them we summarize these results in the next proposition.

\begin{prop}
  The polynomial $g_k(x)=x^k + x^{k-1} + \cdots + x^2 + x - 1 $ is
  irreducible over~$\Q$, its root of the smallest modulus is unique
  and equal to $1/\varphi_k$.

  \label{roots}
\end{prop}

 \begin{figure}[ht]
   \centering
   \scalebox{0.8}{\input{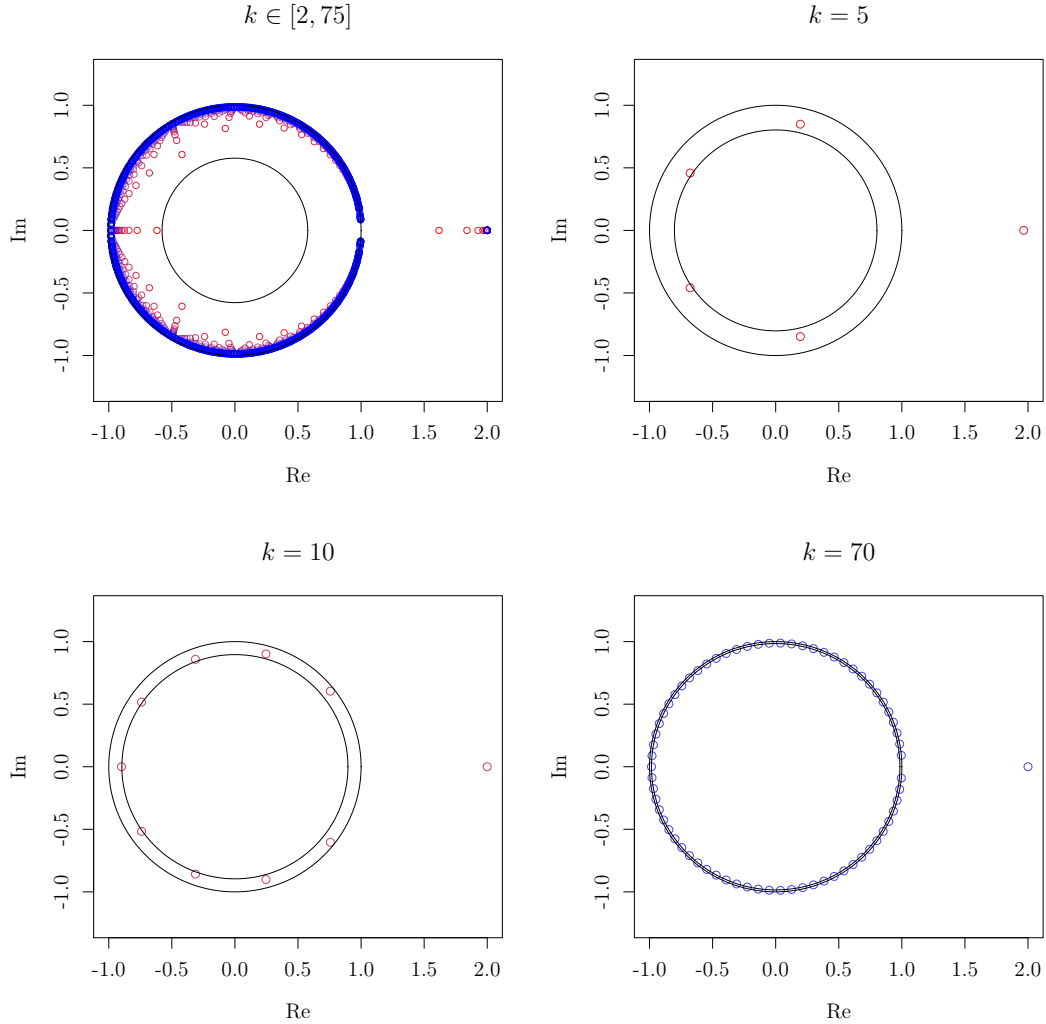}}
   \caption{Roots of the polynomial $x^k - x^{k-1} - \cdots - x^2 - x - 1$ (the negative reciprocal of $g_k(x)$) for certain values of $k$.}
   \label{racines}
 \end{figure}

The next lemma says that both fractions representing $P_k(x)$ and $T_k(x)$ are irreducible.
\begin{lemma}
  The polynomials $\sum_{i = 0}^{k - 2} (i + 1) x^i$ and
  $x^k + x^{k-1} + \cdots + x^2 + x - 1$ are relatively prime; and so are $\sum_{i=0}^{k - 2} (2 i + 2) x^i 
+ \sum_{i=k - 1}^{2 k - 2} (2k - i - 1) x^i$ and $x^k + x^{k-1} + \cdots + x^2 + x - 1$.
\label{le1}
\end{lemma}
\proof
The polynomial  $x^k + x^{k-1} + \cdots + x^2 + x - 1$
is irreducible due to
Proposition~\ref{roots}.
It does not divide $\sum_{i = 0}^{k - 2} (i + 1) x^i$ as it has a greater degree. And
it also cannot divide $\sum_{i=0}^{k - 2} (2 i + 2) x^i 
+ \sum_{i=k - 1}^{2 k - 2} (2k - i - 1) x^i$ as the latter does not have any positive real roots.
\endproof

From Propositions~\ref{pop1}, \ref{roots}, Dubeau's
results~\cite{dubeau}, and Lemma~\ref{le1} we have:

\begin{lemma}
Both generating functions $P_k(x)$ and of $T_k(x)$ have the same and
unique pole of the smallest modulus with multiplicity $2$.  The pole
equals $1/\varphi_k$, where $\varphi_k$ is the generalized golden
ratio.
\label{le2}
\end{lemma}

\bigskip

For our main result of this note we need the 
Theorem 4.1 from~\cite{book}:

\noindent{\textbf{\em Theorem 4.1 from~\cite{book}.}}  {\em Assume
  that a rational generating function $\frac{f(x)}{g(x)}$, with $f(x)$
  and $g(x)$ relatively prime and $g(0) \ne 0$, has a unique pole
  $1/\beta$ of the smallest modulus.
  Then, if the multiplicity of $1/\beta$ is $\nu$, we
  have
  $$[x^n]\frac{f(x)}{g(x)} \sim 
  \nu \frac{(-\beta)^\nu f(1/\beta)}{
    g^{(\nu)}(1/\beta)}
  \beta^n n^{\nu-1}.
$$
}

Both $P_k(x)$ and $T_k(x)$ are rational generating functions, and by Lemmas~\ref{le1} and~\ref{le2}
they fulfill the conditions in the above theorem,
so 
$$
\begin{aligned}
  [x^n]P_k(x) & \sim 2 n \varphi_k^{n+2}\cdot
\frac{
      x  \left( \sum_{i=0}^{k - 2} (i + 1) x^i \right)
}{
  \big( (x^k + x^{k-1} + \cdots + x^2 + x - 1 )^2 \big)''
}\Bigg\vert_{x = 1/\varphi_k}\\
[x^n]T_k(x) & \sim 2 n \varphi_k^{n+2}\cdot
\frac{
      x  \left( \sum_{i=0}^{k - 2} (2 i + 2) x^i
      + \sum_{i=k - 1}^{2 k - 2} (2k - i - 1) x^i \right)
      }{
        \big( ( x^k + x^{k-1} + \cdots + x^2 + x - 1 )^2 \big) ''
} \Bigg\vert_{x = 1/\varphi_k}.
\end{aligned}
$$
The expected value of a random bit in a random word from $\B_n(1^k)$ is
$\frac{[x^n] P_k(x)}{[x^n] T_k(x)}$. Taking the limit, we obtain:
  
\begin{thm}
The expected value of a random bit in a random word from $\B_n(1^k)$
  tends to
  $$
  \frac{k x^{k} - k x^{k-1} - x^{k} + 1}{k x^{k} - k x^{k-1} + x^{2 k} - 3 x^{k} + 2} \bigg\vert_{x = 1/\varphi_k}
  \text{ when } n \to \infty,
  $$
  where $\varphi_k = \lim_{n \to \infty} f_{n+1,k}/f_{n,k}$ is the generalized golden ratio,
  in particular $\varphi_2$ is the golden ratio.
  \label{main}
\end{thm}

See Table \ref{tab2} for some numerical estimations of the result
obtained in the previous theorem.  This result involves the
generalized golden ratio.
More than 20 years ago it was conjectured by Wolfram~\cite{wolf} that
the Galois group of the polynomial $x^k - x^{k-1} - \cdots - x^2 - x -
1$ is the symmetric group $S_k$, and so there is no algebraic
expression for $\varphi_k$ (the root of the largest modulus of this
polynomial) when $k \ge 5$. In case of even or prime $k$ the
conjecture was settled by Martin~\cite{martin}.  Cipu and
Luca~\cite{cipu} showed that $\varphi_k$ cannot be constructed by
ruler and compass for $k \ge 3$.  Nevertheless, good approximations
are available, for instance Hare, Prodinger and Shallit~\cite{hare}
expressed $\varphi_k$ and $1/\varphi_k$ in terms of rapidly converging
series.

\medskip

The generalized golden ratio $\varphi_k$ tends to 2 as 
$k$ grows, and we deduce the following.
\begin{cor}
The limit of the expected bit value of binary words
avoiding $k$ consecutive 1s, whose length tends to infinity,
approaches $1/2$ as $k$ grows:

$$
 \lim_{k \to \infty} \lim_{n \to \infty} \frac{v_{n,k}}{n \cdot |\B_n(1^k)|}
 =
 \frac{1}{2}.
 $$
\label{cor}
\end{cor}

Finally, note that other sets of restricted binary words are counted by the generalized Fibonacci numbers, for instance $q$-decreasing words~\cite{ourfibo} for $q\geq 1$. In this case
every length maximal factor of the form $0^a1^b$ satisfies $a = 0$ or $q \cdot a > b$.
Theorem~\ref{main} and Corollary~\ref{cor}
apply to this case (with the same limit, see~\cite[Corollary 5]{ourfibo})
by setting $k = q + 1$.

\section*{Acknowledgments}

We would like to greatly thank Dietrich Burde, Ted Shifrin, Igor Rivin
and Sil from Mathematics Stack Exchange\footnote{See the original
discussions here
\url{https://math.stackexchange.com/questions/4120185} \\ and here
\url{https://math.stackexchange.com/questions/4125568}.} for
insightful discussions and useful references about the irreducibility
of Fibonacci polynomial which is directly related to
Proposition~\ref{roots}. This work was supported in part by the
project ANER ARTICO funded by Bourgogne-Franche-Comté region (France).

\end{document}

%% file: expected.tex
\begin{tikzpicture}[x=1pt,y=1pt]
\definecolor{fillColor}{RGB}{255,255,255}
\path[use as bounding box,fill=fillColor,fill opacity=0.00] (0,0) rectangle (469.75,180.67);
\begin{scope}
\path[clip] ( 36.00, 24.00) rectangle (234.88,174.67);
\definecolor{drawColor}{RGB}{0,0,0}

\path[draw=drawColor,line width= 0.4pt,line join=round,line cap=round] ( 52.57,169.09) circle (  2.25);

\path[draw=drawColor,line width= 0.4pt,line join=round,line cap=round] ( 61.78, 76.09) circle (  2.25);

\path[draw=drawColor,line width= 0.4pt,line join=round,line cap=round] ( 70.99, 76.09) circle (  2.25);

\path[draw=drawColor,line width= 0.4pt,line join=round,line cap=round] ( 80.19, 64.46) circle (  2.25);

\path[draw=drawColor,line width= 0.4pt,line join=round,line cap=round] ( 89.40, 61.78) circle (  2.25);

\path[draw=drawColor,line width= 0.4pt,line join=round,line cap=round] ( 98.61, 58.37) circle (  2.25);

\path[draw=drawColor,line width= 0.4pt,line join=round,line cap=round] (107.82, 56.55) circle (  2.25);

\path[draw=drawColor,line width= 0.4pt,line join=round,line cap=round] (117.02, 54.95) circle (  2.25);

\path[draw=drawColor,line width= 0.4pt,line join=round,line cap=round] (126.23, 53.79) circle (  2.25);

\path[draw=drawColor,line width= 0.4pt,line join=round,line cap=round] (135.44, 52.83) circle (  2.25);

\path[draw=drawColor,line width= 0.4pt,line join=round,line cap=round] (144.65, 52.06) circle (  2.25);

\path[draw=drawColor,line width= 0.4pt,line join=round,line cap=round] (153.85, 51.41) circle (  2.25);

\path[draw=drawColor,line width= 0.4pt,line join=round,line cap=round] (163.06, 50.87) circle (  2.25);

\path[draw=drawColor,line width= 0.4pt,line join=round,line cap=round] (172.27, 50.40) circle (  2.25);

\path[draw=drawColor,line width= 0.4pt,line join=round,line cap=round] (181.48, 49.99) circle (  2.25);

\path[draw=drawColor,line width= 0.4pt,line join=round,line cap=round] (190.68, 49.64) circle (  2.25);

\path[draw=drawColor,line width= 0.4pt,line join=round,line cap=round] (199.89, 49.32) circle (  2.25);

\path[draw=drawColor,line width= 0.4pt,line join=round,line cap=round] (209.10, 49.05) circle (  2.25);

\path[draw=drawColor,line width= 0.4pt,line join=round,line cap=round] (218.30, 48.80) circle (  2.25);

\path[draw=drawColor,line width= 0.4pt,line join=round,line cap=round] (227.51, 48.57) circle (  2.25);
\end{scope}
\begin{scope}
\path[clip] (  0.00,  0.00) rectangle (469.75,180.67);
\definecolor{drawColor}{RGB}{0,0,0}

\path[draw=drawColor,line width= 0.4pt,line join=round,line cap=round] ( 43.37, 24.00) -- (227.51, 24.00);

\path[draw=drawColor,line width= 0.4pt,line join=round,line cap=round] ( 43.37, 24.00) -- ( 43.37, 18.00);

\path[draw=drawColor,line width= 0.4pt,line join=round,line cap=round] ( 89.40, 24.00) -- ( 89.40, 18.00);

\path[draw=drawColor,line width= 0.4pt,line join=round,line cap=round] (135.44, 24.00) -- (135.44, 18.00);

\path[draw=drawColor,line width= 0.4pt,line join=round,line cap=round] (181.48, 24.00) -- (181.48, 18.00);

\path[draw=drawColor,line width= 0.4pt,line join=round,line cap=round] (227.51, 24.00) -- (227.51, 18.00);

\node[text=drawColor,anchor=base,inner sep=0pt, outer sep=0pt, scale=  1.00] at ( 43.37,  2.40) {0};

\node[text=drawColor,anchor=base,inner sep=0pt, outer sep=0pt, scale=  1.00] at ( 89.40,  2.40) {5};

\node[text=drawColor,anchor=base,inner sep=0pt, outer sep=0pt, scale=  1.00] at (135.44,  2.40) {10};

\node[text=drawColor,anchor=base,inner sep=0pt, outer sep=0pt, scale=  1.00] at (181.48,  2.40) {15};

\node[text=drawColor,anchor=base,inner sep=0pt, outer sep=0pt, scale=  1.00] at (227.51,  2.40) {20};

\path[draw=drawColor,line width= 0.4pt,line join=round,line cap=round] ( 36.00, 29.58) -- ( 36.00,169.09);

\path[draw=drawColor,line width= 0.4pt,line join=round,line cap=round] ( 36.00, 29.58) -- ( 30.00, 29.58);

\path[draw=drawColor,line width= 0.4pt,line join=round,line cap=round] ( 36.00, 57.48) -- ( 30.00, 57.48);

\path[draw=drawColor,line width= 0.4pt,line join=round,line cap=round] ( 36.00, 85.39) -- ( 30.00, 85.39);

\path[draw=drawColor,line width= 0.4pt,line join=round,line cap=round] ( 36.00,113.29) -- ( 30.00,113.29);

\path[draw=drawColor,line width= 0.4pt,line join=round,line cap=round] ( 36.00,141.19) -- ( 30.00,141.19);

\path[draw=drawColor,line width= 0.4pt,line join=round,line cap=round] ( 36.00,169.09) -- ( 30.00,169.09);

\node[text=drawColor,rotate= 90.00,anchor=base,inner sep=0pt, outer sep=0pt, scale=  1.00] at ( 21.60, 29.58) {0.25};

\node[text=drawColor,rotate= 90.00,anchor=base,inner sep=0pt, outer sep=0pt, scale=  1.00] at ( 21.60, 57.48) {0.30};

\node[text=drawColor,rotate= 90.00,anchor=base,inner sep=0pt, outer sep=0pt, scale=  1.00] at ( 21.60, 85.39) {0.35};

\node[text=drawColor,rotate= 90.00,anchor=base,inner sep=0pt, outer sep=0pt, scale=  1.00] at ( 21.60,113.29) {0.40};

\node[text=drawColor,rotate= 90.00,anchor=base,inner sep=0pt, outer sep=0pt, scale=  1.00] at ( 21.60,141.19) {0.45};

\node[text=drawColor,rotate= 90.00,anchor=base,inner sep=0pt, outer sep=0pt, scale=  1.00] at ( 21.60,169.09) {0.50};

\path[draw=drawColor,line width= 0.4pt,line join=round,line cap=round] ( 36.00, 24.00) --
	(234.88, 24.00) --
	(234.88,174.67) --
	( 36.00,174.67) --
	( 36.00, 24.00);
\end{scope}
\begin{scope}
\path[clip] (  0.00,  0.00) rectangle (234.88,180.67);
\definecolor{drawColor}{RGB}{0,0,0}

\node[text=drawColor,rotate= 90.00,anchor=base,inner sep=0pt, outer sep=0pt, scale=  1.00] at ( -2.40, 99.34) {Expected bit value};
\end{scope}
\begin{scope}
\path[clip] ( 36.00, 24.00) rectangle (234.88,174.67);
\definecolor{drawColor}{RGB}{0,0,0}

\path[draw=drawColor,line width= 0.4pt,line join=round,line cap=round] ( 36.00, 44.31) -- (234.88, 44.31);
\end{scope}
\begin{scope}
\path[clip] (270.88, 24.00) rectangle (469.75,174.67);
\definecolor{drawColor}{RGB}{0,0,0}

\path[draw=drawColor,line width= 0.4pt,line join=round,line cap=round] (287.45,169.09) circle (  2.25);

\path[draw=drawColor,line width= 0.4pt,line join=round,line cap=round] (296.66,169.09) circle (  2.25);

\path[draw=drawColor,line width= 0.4pt,line join=round,line cap=round] (305.87,129.23) circle (  2.25);

\path[draw=drawColor,line width= 0.4pt,line join=round,line cap=round] (315.07,126.17) circle (  2.25);

\path[draw=drawColor,line width= 0.4pt,line join=round,line cap=round] (324.28,122.59) circle (  2.25);

\path[draw=drawColor,line width= 0.4pt,line join=round,line cap=round] (333.49,118.36) circle (  2.25);

\path[draw=drawColor,line width= 0.4pt,line join=round,line cap=round] (342.69,116.44) circle (  2.25);

\path[draw=drawColor,line width= 0.4pt,line join=round,line cap=round] (351.90,114.79) circle (  2.25);

\path[draw=drawColor,line width= 0.4pt,line join=round,line cap=round] (361.11,113.42) circle (  2.25);

\path[draw=drawColor,line width= 0.4pt,line join=round,line cap=round] (370.32,112.40) circle (  2.25);

\path[draw=drawColor,line width= 0.4pt,line join=round,line cap=round] (379.52,111.55) circle (  2.25);

\path[draw=drawColor,line width= 0.4pt,line join=round,line cap=round] (388.73,110.83) circle (  2.25);

\path[draw=drawColor,line width= 0.4pt,line join=round,line cap=round] (397.94,110.23) circle (  2.25);

\path[draw=drawColor,line width= 0.4pt,line join=round,line cap=round] (407.15,109.72) circle (  2.25);

\path[draw=drawColor,line width= 0.4pt,line join=round,line cap=round] (416.35,109.27) circle (  2.25);

\path[draw=drawColor,line width= 0.4pt,line join=round,line cap=round] (425.56,108.88) circle (  2.25);

\path[draw=drawColor,line width= 0.4pt,line join=round,line cap=round] (434.77,108.53) circle (  2.25);

\path[draw=drawColor,line width= 0.4pt,line join=round,line cap=round] (443.97,108.23) circle (  2.25);

\path[draw=drawColor,line width= 0.4pt,line join=round,line cap=round] (453.18,107.95) circle (  2.25);

\path[draw=drawColor,line width= 0.4pt,line join=round,line cap=round] (462.39,107.71) circle (  2.25);
\end{scope}
\begin{scope}
\path[clip] (  0.00,  0.00) rectangle (469.75,180.67);
\definecolor{drawColor}{RGB}{0,0,0}

\path[draw=drawColor,line width= 0.4pt,line join=round,line cap=round] (278.24, 24.00) -- (462.39, 24.00);

\path[draw=drawColor,line width= 0.4pt,line join=round,line cap=round] (278.24, 24.00) -- (278.24, 18.00);

\path[draw=drawColor,line width= 0.4pt,line join=round,line cap=round] (324.28, 24.00) -- (324.28, 18.00);

\path[draw=drawColor,line width= 0.4pt,line join=round,line cap=round] (370.32, 24.00) -- (370.32, 18.00);

\path[draw=drawColor,line width= 0.4pt,line join=round,line cap=round] (416.35, 24.00) -- (416.35, 18.00);

\path[draw=drawColor,line width= 0.4pt,line join=round,line cap=round] (462.39, 24.00) -- (462.39, 18.00);

\node[text=drawColor,anchor=base,inner sep=0pt, outer sep=0pt, scale=  1.00] at (278.24,  2.40) {0};

\node[text=drawColor,anchor=base,inner sep=0pt, outer sep=0pt, scale=  1.00] at (324.28,  2.40) {5};

\node[text=drawColor,anchor=base,inner sep=0pt, outer sep=0pt, scale=  1.00] at (370.32,  2.40) {10};

\node[text=drawColor,anchor=base,inner sep=0pt, outer sep=0pt, scale=  1.00] at (416.35,  2.40) {15};

\node[text=drawColor,anchor=base,inner sep=0pt, outer sep=0pt, scale=  1.00] at (462.39,  2.40) {20};

\path[draw=drawColor,line width= 0.4pt,line join=round,line cap=round] (270.88, 29.58) -- (270.88,169.09);

\path[draw=drawColor,line width= 0.4pt,line join=round,line cap=round] (270.88, 29.58) -- (264.88, 29.58);

\path[draw=drawColor,line width= 0.4pt,line join=round,line cap=round] (270.88, 57.48) -- (264.88, 57.48);

\path[draw=drawColor,line width= 0.4pt,line join=round,line cap=round] (270.88, 85.39) -- (264.88, 85.39);

\path[draw=drawColor,line width= 0.4pt,line join=round,line cap=round] (270.88,113.29) -- (264.88,113.29);

\path[draw=drawColor,line width= 0.4pt,line join=round,line cap=round] (270.88,141.19) -- (264.88,141.19);

\path[draw=drawColor,line width= 0.4pt,line join=round,line cap=round] (270.88,169.09) -- (264.88,169.09);

\node[text=drawColor,rotate= 90.00,anchor=base,inner sep=0pt, outer sep=0pt, scale=  1.00] at (256.48, 29.58) {0.25};

\node[text=drawColor,rotate= 90.00,anchor=base,inner sep=0pt, outer sep=0pt, scale=  1.00] at (256.48, 57.48) {0.30};

\node[text=drawColor,rotate= 90.00,anchor=base,inner sep=0pt, outer sep=0pt, scale=  1.00] at (256.48, 85.39) {0.35};

\node[text=drawColor,rotate= 90.00,anchor=base,inner sep=0pt, outer sep=0pt, scale=  1.00] at (256.48,113.29) {0.40};

\node[text=drawColor,rotate= 90.00,anchor=base,inner sep=0pt, outer sep=0pt, scale=  1.00] at (256.48,141.19) {0.45};

\node[text=drawColor,rotate= 90.00,anchor=base,inner sep=0pt, outer sep=0pt, scale=  1.00] at (256.48,169.09) {0.50};

\path[draw=drawColor,line width= 0.4pt,line join=round,line cap=round] (270.88, 24.00) --
	(469.75, 24.00) --
	(469.75,174.67) --
	(270.88,174.67) --
	(270.88, 24.00);
\end{scope}
\begin{scope}
\path[clip] (234.88,  0.00) rectangle (469.75,180.67);
\definecolor{drawColor}{RGB}{0,0,0}

\node[text=drawColor,rotate= 90.00,anchor=base,inner sep=0pt, outer sep=0pt, scale=  1.00] at (232.48, 99.34) {Expected bit value};
\end{scope}
\begin{scope}
\path[clip] (270.88, 24.00) rectangle (469.75,174.67);
\definecolor{drawColor}{RGB}{0,0,0}

\path[draw=drawColor,line width= 0.4pt,line join=round,line cap=round] (270.88,103.01) -- (469.75,103.01);
\end{scope}
\end{tikzpicture}

%% file: freq.tex
\begin{tikzpicture}[x=1pt,y=1pt]
\definecolor{fillColor}{RGB}{255,255,255}
\path[use as bounding box,fill=fillColor,fill opacity=0.00] (0,0) rectangle (202.36,350.51);
\begin{scope}
\path[clip] ( 49.20, 61.20) rectangle (177.16,301.31);
\definecolor{drawColor}{RGB}{0,0,0}

\path[draw=drawColor,line width= 0.4pt,line join=round,line cap=round] ( 53.94, 70.09) circle (  2.25);

\path[draw=drawColor,line width= 0.4pt,line join=round,line cap=round] ( 64.71,174.85) circle (  2.25);

\path[draw=drawColor,line width= 0.4pt,line join=round,line cap=round] ( 75.48,226.71) circle (  2.25);

\path[draw=drawColor,line width= 0.4pt,line join=round,line cap=round] ( 86.25,255.01) circle (  2.25);

\path[draw=drawColor,line width= 0.4pt,line join=round,line cap=round] ( 97.02,271.10) circle (  2.25);

\path[draw=drawColor,line width= 0.4pt,line join=round,line cap=round] (107.79,280.38) circle (  2.25);

\path[draw=drawColor,line width= 0.4pt,line join=round,line cap=round] (118.56,285.74) circle (  2.25);

\path[draw=drawColor,line width= 0.4pt,line join=round,line cap=round] (129.33,288.82) circle (  2.25);

\path[draw=drawColor,line width= 0.4pt,line join=round,line cap=round] (140.10,290.57) circle (  2.25);

\path[draw=drawColor,line width= 0.4pt,line join=round,line cap=round] (150.88,291.56) circle (  2.25);

\path[draw=drawColor,line width= 0.4pt,line join=round,line cap=round] (161.65,292.11) circle (  2.25);

\path[draw=drawColor,line width= 0.4pt,line join=round,line cap=round] (172.42,292.42) circle (  2.25);
\end{scope}
\begin{scope}
\path[clip] (  0.00,  0.00) rectangle (202.36,350.51);
\definecolor{drawColor}{RGB}{0,0,0}

\path[draw=drawColor,line width= 0.4pt,line join=round,line cap=round] ( 53.94, 61.20) -- (161.65, 61.20);

\path[draw=drawColor,line width= 0.4pt,line join=round,line cap=round] ( 53.94, 61.20) -- ( 53.94, 55.20);

\path[draw=drawColor,line width= 0.4pt,line join=round,line cap=round] ( 75.48, 61.20) -- ( 75.48, 55.20);

\path[draw=drawColor,line width= 0.4pt,line join=round,line cap=round] ( 97.02, 61.20) -- ( 97.02, 55.20);

\path[draw=drawColor,line width= 0.4pt,line join=round,line cap=round] (118.56, 61.20) -- (118.56, 55.20);

\path[draw=drawColor,line width= 0.4pt,line join=round,line cap=round] (140.10, 61.20) -- (140.10, 55.20);

\path[draw=drawColor,line width= 0.4pt,line join=round,line cap=round] (161.65, 61.20) -- (161.65, 55.20);

\node[text=drawColor,anchor=base,inner sep=0pt, outer sep=0pt, scale=  1.00] at ( 53.94, 39.60) {2};

\node[text=drawColor,anchor=base,inner sep=0pt, outer sep=0pt, scale=  1.00] at ( 75.48, 39.60) {4};

\node[text=drawColor,anchor=base,inner sep=0pt, outer sep=0pt, scale=  1.00] at ( 97.02, 39.60) {6};

\node[text=drawColor,anchor=base,inner sep=0pt, outer sep=0pt, scale=  1.00] at (118.56, 39.60) {8};

\node[text=drawColor,anchor=base,inner sep=0pt, outer sep=0pt, scale=  1.00] at (140.10, 39.60) {10};

\node[text=drawColor,anchor=base,inner sep=0pt, outer sep=0pt, scale=  1.00] at (161.65, 39.60) {12};

\path[draw=drawColor,line width= 0.4pt,line join=round,line cap=round] ( 49.20, 93.60) -- ( 49.20,292.78);

\path[draw=drawColor,line width= 0.4pt,line join=round,line cap=round] ( 49.20, 93.60) -- ( 43.20, 93.60);

\path[draw=drawColor,line width= 0.4pt,line join=round,line cap=round] ( 49.20,143.40) -- ( 43.20,143.40);

\path[draw=drawColor,line width= 0.4pt,line join=round,line cap=round] ( 49.20,193.19) -- ( 43.20,193.19);

\path[draw=drawColor,line width= 0.4pt,line join=round,line cap=round] ( 49.20,242.99) -- ( 43.20,242.99);

\path[draw=drawColor,line width= 0.4pt,line join=round,line cap=round] ( 49.20,292.78) -- ( 43.20,292.78);

\node[text=drawColor,rotate= 90.00,anchor=base,inner sep=0pt, outer sep=0pt, scale=  1.00] at ( 34.80, 93.60) {0.30};

\node[text=drawColor,rotate= 90.00,anchor=base,inner sep=0pt, outer sep=0pt, scale=  1.00] at ( 34.80,143.40) {0.35};

\node[text=drawColor,rotate= 90.00,anchor=base,inner sep=0pt, outer sep=0pt, scale=  1.00] at ( 34.80,193.19) {0.40};

\node[text=drawColor,rotate= 90.00,anchor=base,inner sep=0pt, outer sep=0pt, scale=  1.00] at ( 34.80,242.99) {0.45};

\node[text=drawColor,rotate= 90.00,anchor=base,inner sep=0pt, outer sep=0pt, scale=  1.00] at ( 34.80,292.78) {0.50};

\path[draw=drawColor,line width= 0.4pt,line join=round,line cap=round] ( 49.20, 61.20) --
	(177.16, 61.20) --
	(177.16,301.31) --
	( 49.20,301.31) --
	( 49.20, 61.20);
\end{scope}
\begin{scope}
\path[clip] (  0.00,  0.00) rectangle (202.36,350.51);
\definecolor{drawColor}{RGB}{0,0,0}

\node[text=drawColor,anchor=base,inner sep=0pt, outer sep=0pt, scale=  1.00] at (113.18, 15.60) {$k$};

\node[text=drawColor,rotate= 90.00,anchor=base,inner sep=0pt, outer sep=0pt, scale=  1.00] at ( 10.80,181.25) {Limit of the expected bit value};
\end{scope}
\end{tikzpicture}